\documentclass[10pt,twocolumn,journal]{IEEEtran}

\usepackage{cite}
\usepackage{graphicx}
\usepackage{latexsym}
\usepackage{amsmath}
\usepackage{amssymb}
\usepackage{algorithm}
\usepackage{algorithmic}

\newtheorem{theorem}{Theorem}
\newtheorem{lemma}{Lemma}

\begin{document}
\title{Efficient Computing Budget Allocation for Simulation-based Optimization with Stochastic Simulation Time
\thanks{This work was supported in part by the National Natural Science
Foundation of China under grants (Nos. 60704008, 60736027, 61174072,
and 90924001), the Specialized Research Fund for the Doctoral
Program of Higher Education (No. 20070003110), the National 111
International Collaboration Project (No. B06002), and the TNList
Cross-Disciplinary Research Funding.}}
\author{Qing-Shan~Jia,~\IEEEmembership{Senior~Member,~IEEE}\thanks{Q.-S. Jia is with Center for Intelligent and Networked
Systems (CFINS), Department of Automation, TNLIST, Tsinghua
University, Beijing 100084, China (Email: jiaqs@tsinghua.edu.cn).}}

\maketitle

\begin{abstract}
The dynamics of many systems nowadays follow not only physical laws
but also man-made rules. These systems are known as discrete event
dynamic systems and their performances can be accurately evaluated
only through simulations. Existing studies on simulation-based
optimization (SBO) usually assume deterministic simulation time for
each replication. However, in many applications such as evacuation,
smoke detection, and territory exploration, the simulation time is
stochastic due to the randomness in the system behavior. We consider
the computing budget allocation for SBO's with stochastic simulation
time in this paper, which has not been addressed in existing
literatures to the author's best knowledge. We make the following
major contribution. The relationship between simulation time and
performance estimation accuracy is quantified. It is shown that when
the asymptotic performance is of interest only the mean value of
individual simulation time matters. Then based on the existing
optimal computing budget allocation (OCBA) method for deterministic
simulation time we develop OCBA for stochastic simulation time
(OCBAS), and show that OCBAS is asymptotically optimal. Numerical
experiments are used to discuss the impact of the variance of
simulation time, the impact of correlated simulation time and
performance estimation, and to demonstrate the performance of OCBAS
on a smoke detection problem in wireless sensor network. The numerical results also show that OCBA for deterministic simulation time is robust even when the simulation time is stochastic.
\end{abstract}


\section{Introduction}\label{secI}

The dynamics of many systems nowadays follow not only physical laws
but also man-made rules. These systems are known as discrete event
dynamic systems (DEDS's). Simulation is usually the only faithful
way to accurately describe the dynamics of such systems. The
performance optimization of these systems then enter the realm of
simulation-based optimization (SBO). Most existing studies on SBO
assume deterministic simulation time for each replications. However,
there exist a large set of DEDS's where the simulation time is
stochastic. Estimating the evacuation time for a building, the smoke
detection time of a wireless sensor network, and the exploration
time of a multi-agent system under a collaborative search policy are
such examples. It is of great practical interest to allocate the
computing budget among designs so that the best design can be found
with high probability. However, to the author's best knowledge, this
problem has not been considered in existing literatures.

We consider this important problem in this paper. Simulation-based
optimization with stochastic simulation time is nontrivial due to the following difficulties. First, simulation-based
performance evaluation. Simulation is usually time-consuming, and
only provides noisy estimations. In order to obtain an accurate
performance estimation, one has to run simulation by infinite times,
which is infeasible in practice. Second, discrete variables. Designs
usually take discrete and finite values. This not only makes
traditional gradient-based search algorithms not applicable, but
also makes the size of the search space increase exponentially fast
when the system scale increases, which is also known as the curse of
dimensionality. Third, the huge number of computing budget
allocations. One usually does not have time to explore all the
allocations to find the optimum. Instead, sequential allocations
that can iteratively improve their performances are of more
practical interest. Fourth, stochastic simulation time. Giving the
total simulation time that is allocated to a design, it is not clear
how many replications can be finished. Thus the resulting
performance estimation accuracy is not clear.

There exist abundant literatures to address the above first three
difficulties. For example, ranking and selection (R\&S) procedures
are typical procedures for SBO. Bechhofer et
al.\cite{BechhoferSantnerGoldsman1995}, Swisher et
al.\cite{SwisherJacobsonYucesan2003}, and Kim and
Nelson\cite{KimNelson2003} provided excellent review of the R\&S
works. Chen\cite{Chen1996}, Chen et al.\cite{ChenChenYucesan2000},
Chen et al.\cite{ChenLinYucesanChick2000}, and Chen and
Y\"{u}cesan\cite{ChenYucesan2005} developed the optimal computing
budget allocation (OCBA) procedure to maximize the probability of
correctly selecting the best design under a given computing budget.
OCBA has been shown to perform asymptotically optimally when the
simulation time is (identically or nonidentically) deterministic.
OCBA has been extended to tackle the case where the deterministic
computing time for one simulation replication is different across
the alternatives\cite{ChenHeFu2006}, to handle
multiple objective functions\cite{TengLeeChew2007,LeeChewTengGoldsman2010}, simulation-based
constraints\cite{LeePujowidiantoLiChenYap2012}, opportunity
cost\cite{HeChickChen2007}, and complexity
preferences\cite{YanZhouChen2012}. A comprehensive introduction to
OCBA is recently available in \cite{ChenLee2011}. Recent good
surveys on other methods for SBO can be found in
\cite{Andradottir1998,Fu2002,SwisherHydenJacobsonSchruben2004,TekinSabuncuoglu2004,ChenHeFuLee2008}.
The above existing literatures assume deterministic simulation time
and do not address the aforementioned difficulty of stochastic
simulation time.

In this paper we consider the computing budget allocation for SBO
with stochastic simulation time and make the following major
contribution. The relationship between the total simulation time and
the accuracy of performance estimation is quantified. It is shown
that when the asymptotic performance is of interest only the mean
value of individual simulation time matters. Then based on OCBA for
deterministic simulation time we develop OCBA for stochastic
simulation time (OCBAS) and show that OCBAS is asymptotically
optimal. Numerical experiments are used to discuss the impact of the
variance of simulation time, the impact of correlated simulation
time and performance estimation, and to demonstrate the performance
of OCBAS on a smoke detection problem in wireless sensor network. The numerical results also show that OCBA for deterministic simulation time is robust even when the simulation time is stochastic.

The rest of this paper is organized as follows. We present the
mathematical problem formulation in section \ref{secPF}, provide the
main results in section \ref{secMR}, show the numerical results in
section \ref{secNR}, and briefly conclude in section \ref{secC}.

\section{Problem Formulation}\label{secPF}

Consider a finite set of designs $\Theta=\{1,\ldots, k\}$. Let $J_i$
be the true performance of design $i$, which can be accurately
evaluated only through infinite number of replications
\begin{equation}
J_i = \lim_{n\rightarrow\infty} \frac{1}{n}\sum_{j=1}^n \hat{J}_i
(\xi_j),\notag
\end{equation}
where $n$ is the number of replications that are used,
\begin{equation}
\hat{J}_i(\xi_j) = J_i + w_i(\xi_j),\notag
\end{equation}
$\xi_j$ represents the randomness in the $j$-th sample path, and
$w_i$ has i.i.d. Gaussian distribution $N(0,\sigma_i^2)$. Let $t_i$
denote the time that is consumed by an individual replication of
design $i$. We assume the simulation is conducted by a digital computer
and thus $t_i$ takes positive integer values and is stochastic.
Let $f_i$ and $F_i$ be the probability density function
(PDF) and cumulative distribution function (CDF) of $t_i$,
respectively. Assume that $\hat{J}_i$ and $t_i$ are mutually
independent. The case when $\hat{J}_i$ and $t_i$ are correlated will
be discussed in section \ref{secNR}.

Giving $T_i$, the number of replications that design $i$ can be
simulated is stochastic, which is denoted as $n_i$. Then we have
\begin{equation}
\sum_{j=1}^{n_i} t_{i,j} \le T_i < \sum_{j=1}^{n_i+1} t_{i,j},\notag
\end{equation}
where $t_{i,j}$ is the simulation time for the $j$-th simulation of
design $i$. Assume that an incomplete simulation does not output any
estimate. When $T_i$ is large, it is reasonable to assume that
$n_i>0$. The estimate of $J_i$ is
\begin{equation}
\bar{J}_i = \frac{1}{n_i} \sum_{j=1}^{n_i}\hat{J}_i (\xi_j).\notag
\end{equation}
We take the Bayesian viewpoint, which means that the estimates
$\bar{J}_1,\ldots, \bar{J}_k$ are given and the true performances
$J_1,\ldots, J_k$ have posterior estimates
$\tilde{J}_1,\ldots,\tilde{J}_k$. Let $g_i$ and $G_i$ be the PDF and
CDF of $\tilde{J}_i$, respectively. Sort the designs from small to
large according to $\bar{J}_i$, and denote the best design as $b$.
Define the probability of correct selection (PCS) as
\begin{equation}
PCS=\Pr\left\{ \tilde{J}_b \le \tilde{J}_i, i=1,\ldots,k \text{ and
} i\neq b \right\}.\notag
\end{equation}
In other words, we are interested in the probability that the
observed best is the truly best. Now we can mathematically formulate
the problem as
\begin{equation}
\max_{T_1,\ldots,T_k} PCS \text{ s.t. } \sum_{i=1}^k T_i = T,\notag
\end{equation}
where $T$ is the total computing budget. In the following discussion
we will refer this problem as P1. In other words, we are looking for
an allocation of the simulation time among the designs so that the
probability of correct selection is maximized. Note that regarding
$t_{i,j}$ as the event occurrence time and $\hat{J}_i(\xi_j)$ as the
reward, then $\{n_i(T_i)\bar{J}_i(T_i)\}$ is a renewal reward
process\cite{Cox1970}. This fact will be used to show that $G_i$ can
be approximated by a Gaussian (Lemma \ref{lemma3}).

\section{Main Results}\label{secMR}

In this section, we address problem P1 in three steps. First, the
relationship between the simulation time $T_i$ and the distribution
of $\tilde{J}_i$ is quantified. Second, an approximation of the PCS
is provided. Then an approximate version of problem P1 is
formulated and denoted as problem P2. Third, OCBAS is
developed and is shown to solve P2 asymptotically optimally.

We start from quantifying the distribution of $n_i$.
\begin{lemma}\label{lemma1}
For any nonnegative integer $c$, we have
\begin{equation}
\Pr\left\{ n_i\ge c \right\} =\left\{ \begin{array}{ll}
                                        (F_i*f_i^{c-1})(T_i), & c\ge 1, \\
                                        1, & c=0,
                                      \end{array}
\right.\notag
\end{equation}
where $[f*g](t)$ represents the convolution of $f$ and $g$, i.e.,
\begin{eqnarray}
[f*g](t) & \equiv & \int_{-\infty}^\infty f(\tau) g(t-\tau)
d\tau;\notag\\
f^a & \equiv & \underbrace{f*\cdots *f}_{a},\notag
\end{eqnarray}
and $f^0$ is the Dirac delta function.
\end{lemma}
\begin{proof}
When $c=0$, we have $\Pr\left\{ n_i\ge 0
\right\}=1$. When $c\ge 1$, we have $\Pr\left\{n_i\ge c\right\} = \Pr\left\{
\sum_{j=1}^c t_{i,j}\le T_i \right\}$.
Then we prove by induction.\\
Step 1. $c=1$. We have $\Pr\left\{ n_i\ge 1 \right\} = \Pr\left\{ t_{i,1} \le T_i\right\} =
F_i (T_i) = (F_i*f^0)(T_i)$.\\
Step 2. Suppose that we have
\begin{equation}
\Pr\left\{n_i \ge c\right\} = \Pr\left\{ \sum_{j=1}^c t_{i,j}\le T_i
\right\} = (F_i *f^{c-1})(T_i).\notag
\end{equation}
Step 1 has provided one such example for $c=1$. We have
\begin{eqnarray}
&&\Pr\left\{n_i\ge c+1\right\}= \Pr\left\{ \sum_{j=1}^{c+1} t_{i,j}\le T_i \right\}\notag\\
&=& \int \Pr\left\{ \sum_{j=1}^c t_{i,j} + t_{i,c+1} \le T_i
\Bigg\vert \sum_{j=1}^c t_{i,j} = x \right\}\notag\\
&&\times\Pr\left\{\sum_{j=1}^c
t_{i,j}=x\right\} dx\notag\\
&=& \int \Pr\left\{ t_{i,c+1}\le T_i-x
\right\}d\Pr\left\{\sum_{j=1}^c t_{i,j}\le x\right\}\notag\\
&=& \int F_i(T_i-x) d(F_i*f_i^{c-1})(x) \notag\\
&=& \int F_i(T_i-x)f_i^c(x) dx = (F_i*f_i^c)(T_i).\notag
\end{eqnarray}
Combining steps 1 and 2 together, we have
\begin{equation}
\Pr\left\{n_i\ge c\right\}=(F_i*f_i^{c-1})(T_i).\notag
\end{equation}
This completes the proof.
\end{proof}

\begin{lemma}\label{lemma2}
\begin{equation}
\Pr\left\{ n_i=c \right\} = \left\{ \begin{array}{ll}
                                      \left[F_i*(f_i^{c-1} - f_i^c) \right](T_i), & c\ge 1, \\
                                      1-F_i(T_i), & c=0.
                                    \end{array}
\right.\notag
\end{equation}
\end{lemma}
\begin{proof}
When $c=0$, we have
\begin{equation}
\Pr\left\{ n_i=0 \right\} = \Pr\left\{ t_{i,1}>T_i \right\} =
1-F_i(T_i).\notag
\end{equation}
When $c\ge 1$, we have
\begin{eqnarray}
\Pr\left\{ n_i=c \right\} &=& \Pr\left\{ n_i\ge c \right\}-
\Pr\left\{ n_i \ge c+1 \right\}\notag\\
&=& (F_i*f_i^{c-1})(T_i) - (F_i*f_i^c)(T_i)\notag\\
&=& \left[ F_i*(f_i^{c-1}- f_i^c) \right](T_i).\notag
\end{eqnarray}
This completes the proof.
\end{proof}

Now we have
\begin{theorem}\label{th1}
$\hspace{-0.1cm}G_i(x)\hspace{-0.1cm}= \hspace{-0.1cm}\sum_{c=1}^\infty \hspace{-0.1cm}\Phi\hspace{-0.1cm}\left(\frac{x-\bar{J}_i}{\sigma_i
/\sqrt{c}}\right) \hspace{-0.1cm}\left[ F_i*(f_i^{c-1}-f_i^c)\right](T_i)$ $ +
G_i^0(x) \left[ 1-F_i(T_i) \right]$, where $\Phi(\cdot)$ is the CDF
of $N(0,1)$ and $G_i^0(x)$ is the a priori CDF of $J_i$.
\end{theorem}
\begin{proof}
We have
\begin{eqnarray}\label{eq1}
&& \hspace{-0.4cm}G_i(x) = \Pr\left\{\tilde{J}_i \le x\right\}\notag\\
&=& \hspace{-0.4cm}\sum_{c=1}^\infty \Pr\left\{ \tilde{J}_i \le x, n_i=c \right\} +
G_i^0(x)\Pr\left\{n_i=0 \right\}\notag\\
&=& \hspace{-0.4cm}\sum_{c=1}^\infty \Pr\left\{ \hspace{-0.1cm}\tilde{J}_i\le x \vert n_i \hspace{-0.1cm}=\hspace{-0.1cm}c\hspace{-0.1cm}
\right\}\hspace{-0.1cm} \Pr\left\{ n_i\hspace{-0.1cm}=\hspace{-0.1cm}c \right\} \hspace{-0.1cm}+\hspace{-0.1cm} G_i^0(x) \Pr\hspace{-0.1cm}\left\{ n_i\hspace{-0.1cm}=\hspace{-0.1cm}0
\right\}.
\end{eqnarray}
Note that
\begin{equation}\label{eq2}
\Pr\hspace{-0.1cm}\left\{\hspace{-0.1cm}\tilde{J}_i \hspace{-0.1cm}\le\hspace{-0.1cm} x \vert n_i\hspace{-0.1cm}=\hspace{-0.1cm}c\hspace{-0.1cm}\right\} \hspace{-0.1cm}=\hspace{-0.1cm} \Pr\hspace{-0.1cm}\left\{\hspace{-0.05cm}
\sum_{j=1}^c \hat{J}_i(\xi_j)/c \le x\hspace{-0.05cm}\right\} \hspace{-0.1cm}=\hspace{-0.1cm}
\Phi\hspace{-0.1cm}\left(\frac{x-\bar{J}_i}{\sigma_i/\sqrt{c}}\right)\hspace{-0.1cm},
\end{equation}
where the first equality follows from the assumption that
$\hat{J}_i$ and $t_i$ are independent, and the second equality
follows from the assumption that $\hat{J}_i(\xi_j)$ are i.i.d.
Gaussian. Combine Eqs. (\ref{eq1}) and (\ref{eq2}) and Lemma
\ref{lemma2}, we then have
\begin{equation}
G_i(x) \hspace{-0.1cm}=\hspace{-0.1cm} \sum_{c=1}^\infty \hspace{-0.1cm}\Phi\hspace{-0.1cm}
\left(\frac{x-\bar{J}_i}{\sigma_i/\sqrt{c}}\right)\hspace{-0.1cm} \left[
F_i\hspace{-0.1cm}*\hspace{-0.1cm}(f_i^{c-1}\hspace{-0.3cm}-\hspace{-0.1cm}f_i^c) \right]\hspace{-0.1cm}(T_i) + G_i^0(x) \left[ 1\hspace{-0.1cm}-\hspace{-0.1cm}F_i(T_i)
\right].\notag
\end{equation}
This completes the proof.
\end{proof}

Theorem \ref{th1} implies that $\tilde{J}_i$ is not Gaussian.
Instead, its CDF $G_i(x)$ is a weighted average of a sequence of
Gaussian CDF's $\Phi\left( \frac{x-\bar{J}_i}{\sigma_i /\sqrt{c}}
\right)$'s, which has equal mean values and decreasing variances. Note that when $T_i \rightarrow \infty$, $\Pr\{n_i=c\}$ is almost zero for most values of $c$ except for $c\approx T_i/\mu_i$, where $\mu_i = \mathbf{E}[t_i]$. In this case $c$ is the expected value of $n_i$.
We have
\begin{lemma}\label{lemma3}
$\lim_{T_i\rightarrow\infty} \left( \sqrt{n_i}\left( \tilde{J}_i -
\bar{J}_i \right)\le x \right) = \Phi\left(x/\sigma_i\right).$
\end{lemma}
\begin{proof}
Following the elementary renewal theorem\cite{Cox1970}, we have
$\lim_{T_i\rightarrow\infty} {\mathbf{E}[n_i]}/{T_i} =
{1}/{\mu_i}.$
Thus when $T_i\rightarrow\infty$, $n_i$ also goes to infinity. Then
Lemma \ref{lemma3} follows naturally from the central limit theorem.
\end{proof}

Lemma \ref{lemma3} implies that when $T_i$ is large, $G_i(x)$ can be
reasonably approximated by $N(\bar{J}_i, \sigma_i^2\mu_i/T_i)$. The performance of the allocation procedure using this approximation will be shown by numerical experiments in section \ref{secNR}.
Following the Bonferroni inequality we have
\begin{equation}
\Pr\left\{ \tilde{J}_b\le \tilde{J}_i, i=1,\ldots,k \text{ and
}i\neq b \right\}\hspace{-0.1cm}\ge \hspace{-0.1cm}1-\hspace{-0.3cm}\sum_{i=1,i\neq b}^k \hspace{-0.2cm}\Pr\left\{ \tilde{J}_b
>\tilde{J}_i \right\}.\notag
\end{equation}
Following the above analysis $G_b(x)$ and $G_i(x)$ can be
approximated by
$\Phi\left(\frac{x-\bar{J}_b}{\sqrt{\sigma_b^2\mu_b/T_b}}\right)$
and $\Phi\left( \frac{x-\bar{J}_i}{\sqrt{\sigma_i^2\mu_i/T_i}}
\right)$, respectively. Then we have
\begin{equation}
\Pr\left\{\tilde{J}_b >\tilde{J}_i\right\} \approx
\int_{-\frac{\delta_{b,i}}{\sigma_{b,i}}}^\infty
\frac{1}{\sqrt{2\pi}}\exp\left\{ -\frac{t^2}{2} \right\}dt,\notag
\end{equation}
where $\delta_{b,i} = \bar{J}_b - \bar{J}_i$ and $\sigma_{b,i}^2=
{\sigma_b^2\mu_b}/{T_b} + {\sigma_i^2 \mu_i}/{T_i}$. Define
the approximate probability of correct selection (APCS) as
\begin{equation}
APCS \equiv 1-\sum_{i=1,i\neq b}^k
\int_{-\frac{\delta_{b,i}}{\sigma_{b,i}}}^\infty
\frac{1}{\sqrt{2\pi}} \exp\left\{ -\frac{t^2}{2} \right\}dt.\notag
\end{equation}
Then problem P1 can be approximated by
\begin{equation}\label{eq3}
\max_{T_1,\ldots,T_k} APCS 
\text{ s.t. }\sum_{i=1}^k T_i = T.\notag
\end{equation}

Denote the above problem as problem P2. Replacing $n_i$ by
$T_i/\mu_i$ in OCBA\cite{ChenLinYucesanChick2000}, we omit the proof and
directly present the following theorem.
\begin{theorem}\label{th2}
Given a total computing time $T$ to be allocated to $k$ competing
designs whose performances are depicted by random variables with
means $J_1,\ldots, J_k$ and finite variances $\sigma_1^2,\ldots,
\sigma_k^2$, and whose individual simulations take random time with
means $\mu_1,\ldots, \mu_k$ and finite variances, as $T\rightarrow
\infty$, the APCS can be asymptotically maximized when\\
(1)
\begin{equation}\label{eq2.1}
\frac{T_i}{T_j} =
\frac{\sigma_i^2\mu_i/\delta_{b,i}^2}{\sigma_j^2\mu_j/\delta_{b,j}^2},
i,j\in\{1,\ldots, k\} \text{ and }i\neq j \neq b;
\end{equation}
(2)
\begin{equation}\label{eq2.2}
T_b = \sqrt{\sigma_b^2 \mu_b \sum_{i=1,i\neq b}^k
\frac{T_i^2}{\sigma_i^2\mu_i}},
\end{equation}
where $T_i$ is the simulation time
allocated to design $i$, $\delta_{b,i} = \bar{J}_b -\bar{J}_i$, and
$\bar{J}_b = \min_i \bar{J}_i$.
\end{theorem}

Note that in practice the values of $J_i$'s, $\sigma_i^2$'s, and
$\mu_i$'s usually are not known a priori, and are replaced by the
sample means and sample variances, respectively. This gives us the
sequential computing budget allocation in Algorithm 1, which is
called OCBA for stochastic simulation time (or OCBAS for short). Note that each iteration in OCBA fixes the the total number of replications and thus takes stochastic time. But each iteration in OCBAS fixes the total simulation time. So the number of replications of a design in each iteration becomes stochastic. Despite this difference, the total simulation time allocated to a design in OCBA and OCBAS are very close. This will be demonstrated by the close performances of the two methods in the next section.

\begin{algorithm}[t]
\caption{Optimal computing budget allocation for stochastic
simulation time (OCBAS)}
\begin{algorithmic}
    \STATE Step 0: Simulate each design by $T_0$ time; $l\leftarrow 0$;
    $T_1^l= T_2^l = \cdots T_k^l = T_0$.
    \STATE Step 1: If $\sum_{i=1}^k T_i \ge T$, stop.
    \STATE Step 2: Increase the total simulation time by $\Delta_T$ and compute
    the new budget allocation $T_1^{l+1},\ldots, T_k^{l+1}$ using
    Theorem \ref{th2}.
    \STATE Step 3: Simulate design $i$ for additional
    $\max\left( 0,T_i^{l+1}-T_i^l \right)$ time, $i=1,\ldots,k$; $l\leftarrow l+1$.
    Go to step 1.
\end{algorithmic}
\end{algorithm}

\section{Numerical Results}\label{secNR}
In this section, we present three groups of numerical experiments to
demonstrate the performance of OCBAS. The first group discusses the
impact of variance of individual simulation time (subsection \ref{secIV}). The second group discusses the impact of the
correlation between individual simulation time and performance
estimation (subsection \ref{secIC}). The third group is a
smoke detection problem (subsection \ref{secSD}). Three methods are considered. First, equal allocation (EA), which equally allocates the simulation time among the designs. Second, OCBA, which iteratively allocates the number of replications among the designs\cite{ChenLinYucesanChick2000}. Third, OCBAS, which iteratively allocates the simulation time among the designs.

\subsection{Impact of Variance of Individual Simulation
Time}\label{secIV}

Consider 10 designs with true performances $J_i=i-1$, $i=1,\ldots,
10$. The performance estimation of all the designs have i.i.d. noise
$N(0,6^2)$. The individual simulation time of each designs are
independent and all have the same variance. We consider two types of distributions of the simulation time. First, uniform distribution. We conduct 10 groups of
experiments to consider 10 values of variances, in which the
individual simulation time takes integer values from $[11-j, 9+j]$
with equal probability, $j=1,\ldots, 10$.  Second, truncated discrete Gaussian distribution. We conduct another 10 groups of experiments, in which the individual simulation time of design $i$ satisfies
$\Pr\{t_i=x\} \propto  \Phi\left((x-i+1.5)/{j}\right) - \Phi\left( (x-i+0.5)/{j}\right), x=1, \ldots, 19;
\Pr\{t_i=x\} = 0, \text{ otherwise.}$
Note that different designs have different truncated discrete Gaussian distributions. Assume that the individual
simulation time and performance estimation are independent. We apply
EA, OCBA ($n_0=5,\Delta_n=10$, which means that each design is observed by 5 replications in the beginning and 10 replications are allocated among the designs in each iteration afterwards), and OCBAS ($T_0=50,\Delta_T=100$, which means that each design is observed using 50 units of time in the beginning and 100 units of simulation time are allocated among the designs in each iteration afterwards)
under $T=1000, \ldots, 10000$. Note that in each iteration of OCBA
we calculate the additional number of simulations that are allocated
to each design, which may take a random simulation time to complete.
This is different from OCBAS, in which we allocate the simulation
time directly. The PCS's are estimated using 10000 replications and
shown in Fig. \ref{figE1}. We make the following remarks.

\begin{figure}
\begin{center}
\includegraphics[width=0.45\textwidth]{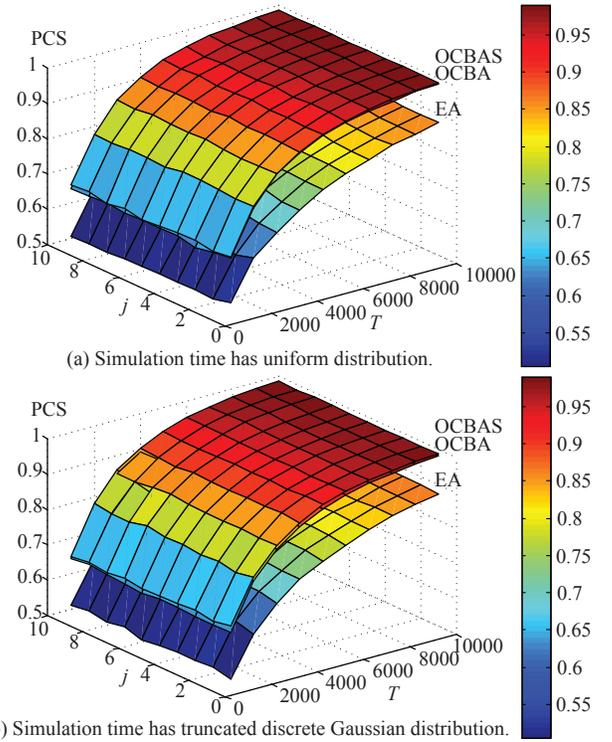}    
\caption{PCS's of EA, OCBA, and OCBAS of subsection \ref{secIV} (averaged over 10000 replications).}  
\label{figE1}                                 
\end{center}                                 
\end{figure}

Remark 1. When the computing budget increases all three methods
achieve higher PCS's. This shows that all three methods can
incrementally improve the PCS's when more computing budget is
available. This salient feature allows sequential allocations of the
computing budget, which are usually preferred over fixed allocations
beforehand in practice.

Remark 2. OCBAS substantially
improves the PCS for a given $T$ (or in other words substantially saves the computing
budget to achieve a given PCS).

Remark 3. OCBA and OCBAS achieve very close performances. The difference between their PCS's are due to randomness of the experiments. This is because the allocation procedures of OCBA
and OCBAS share the same spirit. The difference is that each iteration in OCBA fixes the total number of replications and thus takes stochastic time. But each iteration in OCBAS fixes the total simulation time. Then the total number of replications of a design becomes stochastic.

Remark 4. When the variance of individual simulation time increases,
the PCS's of OCBA and OCBAS do not change much. This is because each design is observed by more and more times when $T$ increases. So the variance of the simulation time does not significantly contribute to the performance estimation when the total simulation time is large, as shown in Lemma \ref{lemma3}. Note that when the total simulation time is small, the variance of the simulation time matters. For example, when $T\le 10$ (the mean value of each individual simulation), a larger variance allows to complete a simulation within $T$ with a larger probability. When the $T>10$, this impact of the variance reduces fast.

\subsection{Correlated Simulation Time and Performance
Estimation}\label{secIC}

Consider 10 designs with true performances $J_i=i-1, i=1,\ldots,
10$. The performance estimation of all the designs have i.i.d. noise
$N(0,6^2)$. The individual simulation time $t_i$ of design $i$ takes values of 5 and 15 with equal probability and is correlated
with its performance observation noise $w_i$ as follows. If $w_i\ge 0$, then $t_i=15$ with probability (w.p.) $p$
and $t_i=5$ w.p. $1-p$, where $0\le p\le 1$ is a given constant. If
$w_i<0$, then $t_i=15$ w.p. $1-p$ and $t_i=5$ w.p. $p$. The value of $p$ indicates the
correlation between $t_i$ and $w_i$. In particular, $p=0$ means that
$t_i$ and $w_i$ are purely negatively correlated; $p=1$ means that
$t_i$ and $w_i$ are purely positively correlated; and $p=0.5$ means
that $t_i$ and $w_i$ are independent. For $p=0,0.1,\ldots, 1.0$ and
$T=1000,\ldots, 10000$, we apply EA, OCBA ($n_0=5,\Delta_n=10$), and
OCBAS ($T_0=50,\Delta_T=100$) and estimate the PCS's by 10000
replications (shown in Fig. \ref{figE4}). Remarks 1-3 also hold in this case. We can also see that the correlation $p$ does not affect PCS much.

\begin{figure}
\begin{center}
\includegraphics[width=0.45\textwidth]{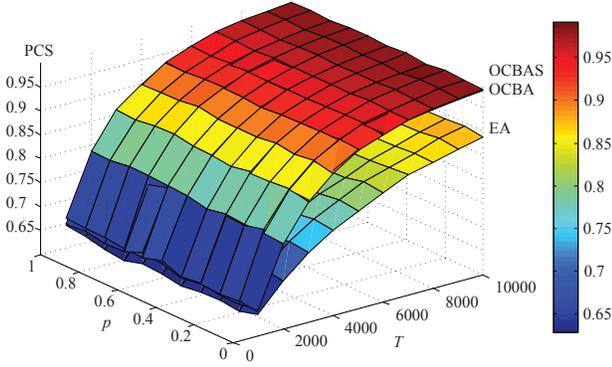}    
\caption{PCS's of EA, OCBA, and OCBAS of subsection \ref{secIC} (averaged over 10000 replications).}  
\label{figE4}                                 
\end{center}                                 
\end{figure}

\subsection{A Smoke Detection Problem}\label{secSD}

We compare three methods on a smoke
detection problem in wireless sensor network. Consider an area of
interest (AoI) with unit size as shown in Fig. \ref{fig1}, which is
discretized into $11\times 11$ grids. A fire may be set at any point
on the grid inside the AoI with equal probability. Once a fire is
set on, it generates a smoke particle within each time slot. A smoke
particle may walk to a neighboring grid in each time slot. There are
at most four such neighboring grids corresponding to four
directions. The probability to walk to one of the four grids is
proportional to its distance to the fire source, i.e.,
\begin{eqnarray}
\Pr\left\{x_{t+1}=x_t +1, y_{t+1} = y_t\right\} &\propto &
d((x_t+1, y_t), (x_0, y_0)),\notag\\
\Pr\left\{x_{t+1}=x_t -1, y_{t+1}=y_t\right\} &\propto & d((x_t-1,
y_t), (x_0, y_0)),\notag\\
\Pr\left\{x_{t+1}=x_t, y_{t+1}=y_t+1\right\} &\propto & d((x_t,
y_t+1), (x_0,y_0)),\notag\\
\Pr\left\{x_{t+1}=x_t, y_{t+1}=y_t-1\right\} &\propto &
d((x_t,y_t-1), (x_0,y_0)),\notag
\end{eqnarray}
where $(x_0,y_0)$ represents the position of the fire source and
$d(\cdot,\cdot)$ represents the distance between two positions. Once
a particle walks to the boundary of AoI, it is bounced back. There
are 3 sensors that can be allocated to the 9 positions marked by
circles in Fig. \ref{fig1}. Once a smoke particle arrives at any of
the three sensors, it is detected. The question is how to allocate
the sensors to minimize the average detection time. It is easy to
show that there are 84 allocations in total. Considering the
symmetries, only 16 allocations need to be considered. The response
time of the 16 designs are evaluated by 100000 independent
replications and shown in Table \ref{tab1}, where the designs are
represented by the positions of the three sensors. Note that in each simulation the response time takes integer values. But the mean values of the response time take positive real numbers. As an example, we show the probability mass function of the response time of the first design (design 1,2,3) in Fig. \ref{figPDF}, which is estimated by 100000 independent replications. Note that in this example we have $\hat{J}_i=t_i$, i.e., the performance estimation and individual simulation time are the same. This violates the assumption used in Theorem \ref{th2}.

\begin{figure}
\begin{center}
\includegraphics[width=0.3\textwidth]{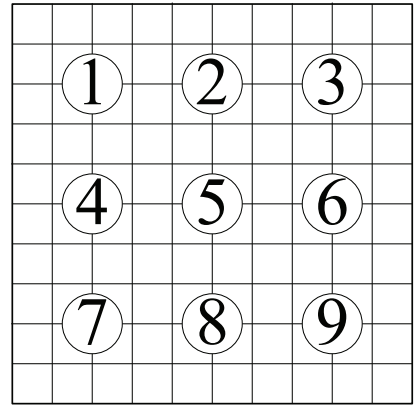}    
\caption{A smoke detection problem in wireless sensor network.}  
\label{fig1}                                 
\end{center}                                 
\end{figure}

\begin{figure}
\begin{center}
\includegraphics[width=0.45\textwidth]{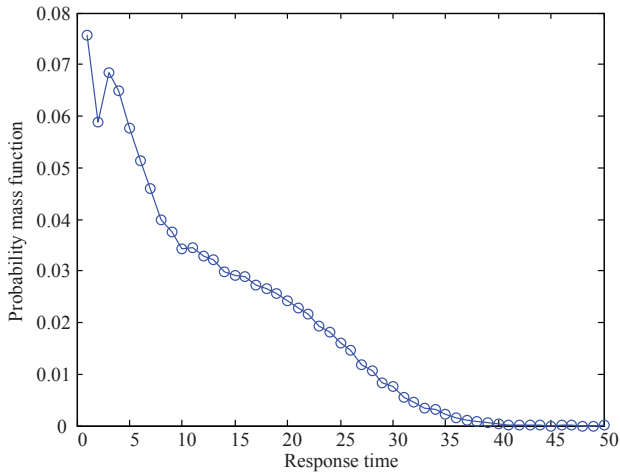}    
\caption{The probability mass function of the response time of the first design, estimated by 100000 replications.}  
\label{figPDF}                                 
\end{center}                                 
\end{figure}

\begin{table}
\caption{True performances of the designs}\label{tab1}\centering
\begin{tabular}[htbp]{rrr|rrr}
\hline\hline Index & Design & Response time & Index & Design & Response time \\
\hline
1 & 1,2,3 & 11.5989 & 9 & 1,3,7 & 8.6777\\
2 & 1,2,4 & 10.6383 & 10 & 1,3,8 & 7.6482\\
3 & 1,2,5 & 9.3776 & 11 & 1,5,6 & 8.0903\\
4 & 1,2,6 & 9.3353 & 12 & 1,5,9 & 8.1355\\
5 & 1,2,7 & 9.7781 & 13 & 1,6,8 & 7.4699\\
6 & 1,2,8 & 8.2390 & 14 & 2,4,5 & 8.5127\\
7 & 1,2,9 & 8.7794 & 15 & 2,4,6 & 7.6968\\
8 & 1,3,5 & 8.6344 & 16 & 2,5,8 & 7.7671\\
\hline\hline
\end{tabular}
\end{table}

The probability of correct selection of the three methods for
$T=1\times 10^4, 2\times 10^4, \ldots, 1\times 10^5$ are evaluated
using 10000 independent replications and shown in Fig.
\ref{figFire}. Since each individual simulation of a design takes
about 10 units of time (as shown in Table \ref{tab1}) to make a
fair comparison between OCBA and OCBAS, we use the following
parameter settings. In OCBA, let $n_0=20$ and $\Delta_n=10$. In OCBAS, let $T_0=200$ and $\Delta_T=100$. Remarks 1-3 also hold in this case. We can see that OCBAS works well even when the performance estimation and individual simulation time are correlated.

\begin{figure}
\begin{center}
\includegraphics[width=0.5\textwidth]{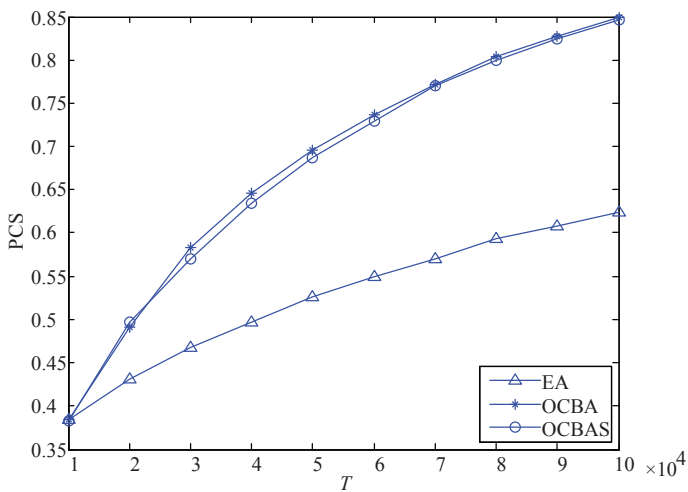}    
\caption{PCS's of EA, OCBA, and OCBAS of subsection \ref{secSD} (averaged over 10000 replications).}  
\label{figFire}                                 
\end{center}                                 
\end{figure}

\section{Conclusion}\label{secC}
In this paper, we consider the computing budget allocation for SBOs
with stochastic simulation time and develop OCBAS to provide the
allocation asymptotically optimally. The performance of OCBAS is
demonstrated through two groups of academic examples and a smoke
detection problem in wireless sensor network. The numerical results also show that OCBA for deterministic simulation time is robust even when the simulation time is stochastic. Note that the
asymptotically optimal allocation of computing budget only depends
on the mean value of the stochastic simulation time. Note that
though we assume the performance estimate $\hat{J}_i$ and the
simulation time for an individual replication $t_i$ are independent
in section \ref{secMR}, the numerical results in section \ref{secNR}
show that OCBAS performs well when $\hat{J}_i$ and $t_i$ are
correlated. Note that Lemma \ref{lemma3} shows that the performance
estimator under stochastic simulation time can be well approximated
by the performance estimator under deterministic simulation time.
Replacing $n_i$ by $T_i/\mu_i$, OCBAS can be obtained
straightforwardly from OCBA. Using Lemma \ref{lemma3}, it is
possible to extend OCBAS to handle multiple objective functions,
simulation-based constraints, opportunity cost, and complexity
preferences, following its according extensions in OCBA. That will
be important future work. Note that when $N$ parallel computers are available, the total computing budget we can use will be improved from $T$ to $NT$. Both OCBA and OCBAS can be extended to this situation. But if each computer can simulate only a specific design, i.e., $N=k$, we usually have a constraint on the decision making time that is $\max_i{T_i}$. How to allocate computing budget according to this constraint is an interesting future research topic. We hope this work brings insights on
addressing SBOs with stochastic simulation time in general.

\section*{Acknowledgments}
The author would like to thank the editor, the associate editor, and
the anonymous reviewers for their constructive comments on earlier
versions of this paper.


\bibliographystyle{IEEEtran}
\bibliography{IEEEabrv,KCMDP}

\end{document}